\documentclass{amsart}
\usepackage[margin=3cm]{geometry}
\usepackage{amsmath,amsfonts,amssymb,amsthm}
\usepackage{graphics}
\usepackage{epsfig, esint}
\usepackage{graphics,color}
\usepackage{paralist}

\RequirePackage[colorlinks,citecolor=blue,urlcolor=blue]{hyperref}

   \definecolor{labelkey}{gray}{.8}
   \definecolor{refkey}{gray}{.8}

\providecommand{\dx}{\, \mathrm{d} x}

\providecommand{\dr}{\, \mathrm{d} r}

\providecommand{\R}{\mathbb{R}}

\newcommand{\substep}[1]{\medskip\noindent\textit{Substep #1. }}

\newcommand{\ignore}[1]{}

\newtheorem{definition}{Definition}
\newtheorem{proposition}{Proposition}
\newtheorem{theorem}{Theorem}
\newtheorem{remark}{Remark}
\newtheorem{lemma}{Lemma}
\newtheorem{corollary}{Corollary}

\providecommand{\PfStart}[1]{\newcounter{#1}\refstepcounter{#1}} 
\providecommand{\PfStep}[2]{ \ifnum\value{#1}=1
\else\medskip\fi{\sc Step }\arabic{#1}\label{#2}\refstepcounter{#1}.} 

\author[P. Bella]{Peter Bella}
\author[M. Sch\"affner]{Mathias Sch\"affner}
\address{Technische Universit\"at Dortmund, Fakult\"at f\"ur Mathematik\\
 Vogelpothsweg 87,44227 Dortmund, Germany.}
\email{peter.bella@udo.edu, mathias.schaeffner@tu-dortmund.de}

\title[Boundedness for $p$-Laplacian]{Local boundedness for $p$-Laplacian with degenerate coefficients} 

\begin{document}
\maketitle


\begin{abstract}
We study local boundedness for subsolutions of nonlinear nonuniformly elliptic equations whose prototype is given by $\nabla \cdot (\lambda |\nabla u|^{p-2}\nabla u)=0$, where the variable coefficient $0\leq\lambda$ and its inverse $\lambda^{-1}$ are allowed to be unbounded. Assuming certain integrability conditions on $\lambda$ and $\lambda^{-1}$ depending on $p$ and the dimension, we show local boundedness. Moreover, we provide counterexamples to regularity showing that the integrability conditions are optimal for every $p>1$.
\end{abstract}

\section{Introduction}

In this note, we study local boundedness of weak (sub)solutions of non-uniformly elliptic quasi-linear equations of the form
\begin{equation}\label{eq}
\nabla \cdot a(x,\nabla u)=0\qquad\mbox{in $\Omega$},
\end{equation}
where $\Omega\subset\R^d$ with $d\geq2$ and $a:\Omega\times\R^d\to\R^d$ is a Caratheodory function. The main example that we have in mind are $p$-Laplace type operators with variable coefficients, that is, there exist $p>1$ and $A:\Omega\to\R^{d\times d}$ such that $a(x,\xi)=A(x)|\xi|^{p-2}\xi$ for all $x\in\Omega$ and $\xi\in\R^d$. In order to measure the ellipticity of $a$, we introduce for fixed $p>1$
\begin{equation}\label{def1} 
 \lambda(x) := \inf_{\xi \in \R^d \setminus \{ 0 \}} \frac{a(x,\xi)\cdot \xi}{|\xi|^p} \qquad \mu(x) := \sup_{\xi \in \R^d \setminus \{ 0 \}} \frac{|a(x,\xi)|^p}{(a(x,\xi)\cdot \xi)^{p-1}}
\end{equation}
and suppose that $\lambda$ and $\mu$ are nonnegative. In the uniformly elliptic setting, that is that there exists $0<m\leq M<\infty$ such that $m\leq \lambda\leq\mu\leq M$ in $\Omega$, solution to \eqref{eq} are locally bounded, H\"older continuous and even satisfy Harnack inequality, see e.g.\ classical results of Ladyzhenskaya \& Ural'tseva, Serrin and Trudinger \cite{LU68,T67}. 

\smallskip

In this contribution, we are interested in a nonuniformly elliptic setting and assume that $\lambda^{-1}\in L^t(\Omega)$ and $\mu\in L^s(\Omega)$ for some integrability exponents $s$ and $t$. In \cite{BS19a}, we studied this in the case of linear nonuniformly elliptic equations, that is $a(x,\xi)=A(x)\xi$  corresponding to the case $p=2$, and showed local boundedness and Harnack inequality for weak solutions of \eqref{eq} provided it holds $\frac1s+\frac1t<\frac2{d-1}$. The results of \cite{BS19a} improved classical findings of Trudinger \cite{T71,T73} (see also \cite{MS68}) from the 1970s and are optimal in view of counterexamples constructed by Franchi et. al. in \cite{FSS98}. In this manuscript we extend these results to the more general situation of quasilinear elliptic equation with $p$-growth as described above. More precisely, we show

\begin{theorem}\label{T:1}
 Let $d \ge 2, p >1$, and let $\Omega \subset \R^d$. Moreover, let $s\in[1,\infty]$ and $t\in(1/(p-1),\infty]$ satisfy
 \begin{equation}\label{pqcond}
  \frac 1s + \frac 1t < \frac p{d-1}.
 \end{equation}
 Let $a : \Omega \times \R^d \to \R^d$ be a Caratheodory function with $a(\cdot,0)\equiv0$ such that $\lambda$ and $\mu$ defined in~\eqref{def1} satisfy $\mu \in L^s(\Omega)$ and $\frac 1\lambda \in L^t(\Omega)$. Then any weak subsolution of \eqref{eq} is locally bounded from above in $\Omega$. 
 \end{theorem}

\begin{remark}
Note that Theorem~\ref{T:1}, restricted to the case $p=2$ recovers the local boundedness part of \cite[Theorem~1.1]{BS19a}. \end{remark}

\begin{remark}
In \cite{CMM18}, Cupini, Marcellini and Mascolo studied local boundedness of local minimizer of nonuniformly elliptic variational integrals of the form $\int_\Omega f(x,\nabla v)\,dx$ where $f$ satisfies
\begin{equation}\label{growth:nonuniformpq}
\lambda(x)|\xi|^p\leq f(x,\xi)\leq \mu(x)+\mu(x)|\xi|^q\qquad\mbox{with $\lambda^{-1}\in L^t(\Omega)$ and $\mu\in L^s(\Omega)$}. 
\end{equation}
They proved local boundedness under the relation $\frac1{p t}+\frac1{q s}+\frac1p-\frac1q<\frac1d$ (see also \cite{BCM20} for related results). Considering the specific case $f(x,\xi)=\lambda(x)|\xi|^p$, the result of \cite{CMM18} implies local boundedness of solutions to $\nabla \cdot(\lambda(x)|\nabla u|^{p-2}\nabla u)=0$ provided $\lambda^{-1}\in L^t(\Omega)$ and $\lambda\in L^s(\Omega)$ with $\frac1{s}+\frac1{t}<\frac{p}d$, which is more restrictive compared to assumption \eqref{pqcond} in Theorem~\ref{T:1}. It would be interesting to investigate if the methods of the present paper can be combined with the ones of \cite{CMM18} to obtain local boundedness for minimizer of functionals satisfying \eqref{growth:nonuniformpq} assuming $\frac1{p t}+\frac1{q s}+\frac1p-\frac1q<\frac1{d-1}$. Note that in the specific case $s=t=\infty$, this follows from \cite{HS19}.
\end{remark}
The proof of Theorem~\ref{T:1} is presented in Section~\ref{sec:positive} and follows a variation of the well-known Moser-iteration method. The main new ingredient compared to earlier works \cite{T71,CMM18} lies in an optimized choice of certain cut-off functions -- an idea that we first used in \cite{BS19a} for linear nonuniformly elliptic equations (see also \cite{AD,BS4,zhang} for recent applications to linear parabolic equations).  

As mentioned above, an example constructed in \cite{FSS98} shows that condition \eqref{pqcond} is optimal for the conclusion of Theorem~\ref{T:1} in the case $p=2$. In the second main result of this paper, we show -- building on the construction of \cite{FSS98} --  that condition \eqref{pqcond} is optimal for the conclusion of Theorem~\ref{T:1} for all $p\in(1,\infty)$. More precisely, we have
\begin{theorem}\label{T:negative}
Let $d\geq3$ and $1+\frac{1}{d-2}<p<\infty$ and let $s\geq1$ and $t>\frac{1}{p-1}$ be such that $\frac1s+\frac1t\geq \frac{p}{d-1}$ and $\frac{1}{1+1/t}p<d-1$. Then there exists $\lambda:B(0,1)\to(0,\infty)$ satisfying $\lambda\in L^s(B_1)$ and $\lambda^{-1}\in L^t(B_1)$ and an unbounded weak subsolution of
\begin{equation}\label{eq:negative}
-\nabla \cdot (\lambda |\nabla v|^{p-2}\nabla v)=0
\end{equation}  
in $B(0,1)$. Moreover, the same conclusion is valid for $d\geq3$, $1<p\leq 1+\frac{1}{d-2}$ and $s\geq1$ and $t>\frac{1}{p-1}$ satisfying the strict inequalities $\frac1s+\frac1t> \frac{p}{d-1}$ and $\frac{t}{t+1}p<d-1$.
\end{theorem}
In particular, we see that condition \eqref{pqcond} is sharp on the scale of Lebesgue-integrability for the conclusion of Theorem~\ref{T:1}. We note that in the particularly interesting case $p=2$ and $d=3$ the construction in Theorem~\ref{T:negative} fails in the critical case $\frac1s+\frac1t=\frac{p}{d-1}$, see \cite{AD} for counterexamples to local boundedness for related problems in $d=3$.

Let us now briefly discuss a similar but different instance of non-uniform ellipticity  which is one of the many areas within the Calculus of Variations, where G.\ Mingione made significant contributions. Consider variational integrals
\begin{equation}\label{eq:int}
\int_\Omega F(x,\nabla u)\,dx,
\end{equation}
where the integrand $F$ satisfies $(p,q)$ growth conditions of the form 
\begin{equation}\label{growth:pq}
|\xi|^p\lesssim F(x,\xi)\lesssim 1+|\xi|^q\qquad1<p\leq q<\infty,
\end{equation} 
which where first systematically studied by Marcellini in \cite{Mar89,Mar91}; see also the recent reviews \cite{Min06,MR21}. The focal point in the regularity theory for those functionals is to obtain Lipschitz-bounds on the minimizer. Indeed, once boundedness of $|\nabla u|$ is proven the unbalanced growth in \eqref{growth:pq} becomes irrelevant and there is a huge literature dedicated to Lipschitz estimates under various assumptions on $F$, see e.g.\ the interior estimates \cite{BM20,BS19c,BS22,ELM02} in the autonomous case, \cite{BCM18,CM15,CMMP21,DM19,DM21,DM22,EMM19,HO21} in the non-autonomous case, \cite{BDMS,DP22} for Lipschitz-bounds at the boundary, and also examples where the regularity of minimizer fail \cite{BDS20,G87,Mar91,ELM04,FMM04}. Finally, we explain a link between functionals with $(p,q)$-growth and (linear) equations with unbounded coefficients. Consider the autonomous case that $F(x,\xi)=F(\xi)$ and let $u\in W^{1,p}(\Omega)$ be a local minimizer of \eqref{eq:int}. Linearizing the corresponding Euler-Largrange equation yield (formally)
$$
\nabla \cdot D^2F(\nabla u)\nabla \partial_i u=0.
$$  
Assuming $(p,q)$-growth with $p=2$ of the form $|\zeta|^2\lesssim D^2F(\xi)\zeta\cdot\zeta\lesssim (1+|\xi|)^{q-2}|\zeta|^2$ implies that $|D^2F(\nabla u)|\in L_{\rm loc}^{\frac{2}{q-2}}(\Omega)$. Hence condition \eqref{pqcond} with $p=2$ yield local boundedness of $\partial_i u$ if $\frac{q-2}{2}<\frac{2}{d-1}$, which is the currently best known general bound ensuring Lipschitz-continuity of local minimizer of \eqref{eq:int} - this reasoning was made rigorous in \cite{BS19c} for $p\geq2$ (see also \cite{BS22} for the case $p\in(1,\infty)$).

\section{Local boundedness, proof of Theorem~\ref{T:1}}\label{sec:positive}

Before we prove Theorem~\ref{T:1}, we introduce the notion of solution that we consider here. 

\begin{definition}\label{def:solution}
Fix a domain $\Omega\subset\R^d$ and a Caratheodory function $a:\Omega\times\R^d\to\R^{d}$ such that for a fixed $p\in(1,\infty)$ the functions $\lambda,\mu\geq0$ given in \eqref{def1} satisfy $\frac1\lambda\in L^{\frac{1}{p-1}}(\Omega)$ and $\mu\in L^1(\Omega)$. The spaces $H_0^{1,p}(\Omega,a)$ and $H^{1,p}(\Omega,a)$ are respectively defined as the completion of $C_c^{1}(\Omega)$ and $C^{1}(\Omega)$ with respect to the norm $\|\cdot\|_{H^{1,p}(\Omega,a)}$, where
\begin{equation*}
\|u\|_{H^{1,p}(\Omega,a)}:=\biggl(\int_\Omega \lambda |\nabla u|^p+\mu |u|^p\,dx\biggr)^\frac1p.
\end{equation*}
We call $u$ a weak solution (subsolution, supersolution) of \eqref{eq} in $\Omega$ if and only if $u\in H^{1,p}(\Omega,a)$ and
\begin{equation}\label{def:harmonic}
 \forall \phi\in H_0^{1,p}(\Omega,a),\, \phi\geq0:\qquad \mathcal A(u,\phi)=0\quad (\leq0,\geq0),\quad\mbox{where}\quad\mathcal A(u,\phi):=\int_\Omega a(x,\nabla u) \cdot \nabla \phi\,dx.
\end{equation}
Moreover, we call $u$ a local weak solution of \eqref{eq} in $\Omega$ if and only if $u$ is a weak solution of \eqref{eq} in $\Omega'$ for every bounded open set $\Omega'\Subset\Omega$. Throughout the paper, we call a solution (subsolution, supersolution) of \eqref{eq} in $\Omega$ $a$-harmonic ($a$-subharmonic, $a$-superharmonic) in $\Omega$. 
\end{definition}
The above definitions generalize the concepts of weak solutions and the spaces $H^1(\Omega,a)$ and $H^1_0(\Omega,a)$ discussed by Trudinger \cite{T71,T73} in the linear case, that is $a(x,\xi)=A(x)\xi$. We stress that the condition $\lambda^{-1}\in L^\frac1{p-1}(\Omega)$ and H\"older inequality imply
$$
\|\nabla u\|_{L^1(\Omega)}\leq \|\lambda^{-1}\|_{L^\frac1{p-1}(\Omega)}\biggl(\int_{\Omega}\lambda |\nabla u|^p\biggr)^\frac1p\leq\|\lambda^{-1}\|_{L^\frac1{p-1}(\Omega)}\|u\|_{H^{1,p}(\Omega,a)}
$$
and thus, we have that $W^{1,1}(\Omega)\subset H^{1,p}(\Omega,a)$, where we use that by the same computation as above it holds $\|u\|_{L^1(\Omega)}\leq\|\mu^{-1}\|_{L^\frac1{p-1}(\Omega)}\|u\|_{H^{1,p}(\Omega,a)}$ and that by definition we have $\lambda\leq\mu$. From this, we also deduce that the elements of $H^{1,p}(\Omega,a)$ are strongly differentiable in the sense of \cite{GT}. In particular this implies that there holds a chain rule in the following sense
\begin{remark}
 Let $g:\R\to\R$ be uniformly Lipschitz-continuous with $g(0)=0$ and consider the composition $F:=g(u)$. Then, $u\in H_0^{1,p}(\Omega,a)$ (or $\in H^{1,p}(\Omega,a)$) implies $F \in H_0^{1,p}(\Omega,a)$ (or $\in H^{1,p}(\Omega,a)$), and it holds $\nabla F=g'(u)\nabla u$ a.e.\ (see e.g.\ \cite[Lemma 1.3]{T73}). In particular,  if $u$ satisfies $u\in H^{1,p}(\Omega,a)$ (or $\in H^{1,p}(\Omega,a)$) then also the truncations
\begin{equation*}
 u_+:=\max\{u,0\};\quad u_{-}:=-\min\{u,0\}
\end{equation*}
satisfy $u_+,u_-\in H^{1,p}(\Omega,a)$ (or $\in H^{1,p}(\Omega,a)$).
\end{remark}
%

Now we come to the local boundedness from above for weak subsolutions of \eqref{eq}.

\begin{theorem}\label{T:est}
 Let $d \ge 3$, $\Omega\subset \R^d$ and $p\in(1,\infty)$. Moreover, let $s\in[1,\infty]$ and $t\in(\frac1{p-1},\infty]$ satisfy \eqref{pqcond}. Let $a : \Omega \times \R^d \to \R^d$ be a Caratheodory function with $a(\cdot,0)\equiv0$ such that $\lambda$ and $\mu$ defined in~\eqref{def1} satisfy $\mu \in L^s(\Omega)$ and $\frac 1\lambda \in L^t(\Omega)$. Then, there exists $c=c(d,p,s,t)\in[1,\infty)$ such that for any weak subsolution $u$ of \eqref{eq} and for any ball $B_{R}\subset \Omega$ it holds
 \begin{equation*}
  \sup_{B_{R/2}} u 
  \le 
 c \Lambda(B_R)^{\frac1p\frac1\delta}    \|u_+\|_{\underline W^{1,\frac{1}{1+1/t}p}(B_R)}   
 \end{equation*}
 where $\Lambda(S) := \bigl( \fint_S \mu^{s} \bigr)^{1/s} \bigl( \fint_S \lambda^{-t} \bigr)^{1/t}$, $\|v\|_{\underline W^{1,\gamma}(B_r)}:=r^{-\frac{d}\gamma}\|v\|_{L^\gamma(B_r)}+r^{1-\frac{d}\gamma}\|\nabla v\|_{L^\gamma(B_r)}$ for all $\gamma\geq1$ and $r>0$; and $\delta := \frac1{s_*}-(\frac1p-\frac1{pt})>0$ (see Lemma~\ref{lm1} for the definition of $s_*$). Moreover, in the case $1+\frac1t<\frac{p}{d-1}$, there exists $c=c(d,p,t)\in[1,\infty)$ such that
  \begin{equation*}
  \sup_{B_{R/2}} u \le c\|u_+\|_{\underline W^{1,\frac{1}{1+1/t}p}(B_R)}.   
 \end{equation*}
\end{theorem}

In the two-dimensional case, we have the following
\begin{proposition}\label{P:2d}
Let $\Omega\subset \R^2$ and $p\in(1,\infty)$. Let $a : \Omega \times \R^d \to \R^d$ be a Caratheodory function with $a(\cdot,0)\equiv0$ such that $\lambda$ and $\mu$ defined in~\eqref{def1} satisfy $\mu \in L^1(\Omega)$ and $\frac 1\lambda \in L^\frac1{p-1}(\Omega)$. Then, there exists $c=c(d,p)\in[1,\infty)$ such that for any weak subsolution $u$ of \eqref{eq} and for any ball $B_{R}\subset \Omega$ it holds
  \begin{equation*}
  \sup_{B_{R/2}} u \le c\|u_+\|_{\underline W^{1,1}(B_R)}.   
 \end{equation*}
\end{proposition}

Before we proof Theorem~\ref{T:est} and Proposition~\ref{P:2d}, we show that they imply the claim of Theorem~\ref{T:1}

\begin{proof}[Proof of Theorem~\ref{T:1}]
In view of Theorem~\ref{T:est} and Proposition~\ref{P:2d} it remains to show that for any weak subsolution $u$ of \eqref{eq} and for any ball $B_{R}\subset \Omega$ it holds $\|u_+\|_{\underline W^{1,\frac{t}{t+1}p}(B_R)}<\infty$. This is a consequence of H\"older inequality and the concept of weak subsolution, see Definition~\ref{def:solution}. Indeed, we have
$$
\biggl(\int_{B_R}(|u|+|\nabla u|)^\frac{tp}{t+1}\biggr)^\frac{t+1}t\leq \biggl(\int_{B_R}\lambda^{-t}\biggr)^\frac1t\int_{B_R}\lambda(|u|+|\nabla u|)^p<\infty,
$$ 
where the right-hand side is finite since $u\in H^{1,p}(\Omega,a)$ (note that $\lambda\leq \mu$ by definition).
\end{proof}

For the proof of Theorem~\ref{T:est}, we need a final bit of preparation, namely the following optimization lemma
\begin{lemma}[Radial optimization]\label{lm1}
Let $d \ge 3$, $p > 1$, $s > 1$, and let $s_* := \max\{1,\big( \frac 1p \bigl( 1 - \frac 1s\bigr) + \frac 1{d-1} \big)^{-1}\}$.
For $\frac 12 \le \rho < \sigma \le 2$, let $v \in W^{1,s_*}(B_\sigma)$ 
and $\mu \in L^s(B_\sigma)$, $\mu \ge 0$, be such that $\mu |v|^p \in L^1(B_\sigma)$. Then there exists $c=c(d,p,s)$ such that 
\begin{equation}\nonumber
 J(\rho,\sigma,v) := \inf \biggl\{ \int_{B_\sigma} \mu |v|^p |\nabla \eta|^p \dx : \eta \in C^1_0(B_\sigma), \eta \ge 0, \eta = 1 \textrm{ in } B_\rho \biggr\}
\end{equation}
satisfies
\begin{equation}\nonumber
 J(\rho,\sigma,v) \le c(\sigma-\rho)^{
 -\frac{pd}{d-1}} 
 \|\mu\|_{L^s(B_\sigma \setminus B_\rho)} \bigl( \|\nabla v\|^p_{L^{s_*}(B_\sigma \setminus B_\rho)} + \rho^{-p} \|v\|^p_{L^{s_*}(B_\sigma \setminus B_\rho)} \bigr). 
\end{equation}
\end{lemma}
Lemma~\ref{lm1} generalizes \cite[Lemma~2.1]{BS19a} from $p=2$ to $p>1$ and we provide a proof in the appendix. 

\begin{proof}[Proof of Theorem~\ref{T:est}]\PfStart{pf1} 

By standard scaling and translation arguments it suffices to suppose that $B_1\Subset \Omega$ and $u$ is locally bounded in $B_\frac{1}2$. Hence, we suppose from now on that $B_1\Subset \Omega$. In Steps~1--4 below, we consider the case $s>1$. We first derive a suitable Caccioppoli-type inequality for powers of $u_+$ (Step~1) and perform a Moser-type iteration (Steps~2--4). In Step~5, we consider the case $1+\frac1t<\frac{p}{d-1}$ which includes the case $s=1$.

\PfStep{pf1}{pf1s1} Caccioppoli  inequality. \\
Assuming $B \subset \Omega$, for any cut-off function $\eta \in \mathcal C^1_0(B)$, $\eta \ge 0$ and any $\beta \ge 1$, there holds 
\begin{equation}\label{Caccioppoli}
  \int \eta^p \lambda(x) u_+^{\beta -1} |\nabla u_+|^p \le 
  \biggl( \frac p \beta \biggr)^{p} \int u_+^{p+\beta-1} \mu(x) |\nabla \eta|^p.
\end{equation}

For $\beta \ge 1$, we use the weak formulation~\eqref{def:harmonic} with $\phi:=\eta^p u_+^{\beta}$: 
\footnote{Rigorously, we are a priori not allowed to test with $u^\beta$. Instead, for $N \ge 1$ one should modify $u^\beta$ by replacing $u^\beta$ with affine $\alpha N^{\alpha-1}u - (\alpha-1)N^\beta$ in the set $u \ge N$, obtain the conclusion by testing the weak formulation with this modified function, and subsequently sends $N \to \infty$ -- for details, see~\cite[Page 460]{BS19a}.}
\begin{equation*}
 \int a(x,\nabla u) \cdot \nabla (\eta^p u_+^\beta ) \le 0.
\end{equation*}
We have $\int (a(x,\nabla u) - a(x,\nabla u_+))\cdot \nabla(\eta^p u_+) = 0$, so that we were able to replace $u$ with $u_+$ inside $a(x,\cdot)$. Applying Leibniz rule we get from the previous display
\begin{equation}\label{eq01}
 \beta \int \eta^p u^{\beta -1} a(x,\nabla u) \cdot \nabla u \le 
 - \int p \eta^{p-1} u^{\beta} a(x,\nabla u) \cdot \nabla \eta,
\end{equation}
where to simplify the notation for the rest of this proof we write $u$ instead of $u_+$. 
Using definition of $\mu$ in~\eqref{def1} in form of $|a(x,\xi)| \le \mu(x)^{\frac 1p} (a(x,\xi)\cdot \xi)^{\frac{p-1}p}$ for any $\xi \in \R^d$ (in fact we use \eqref{def1} for $\xi\neq0$ and for $\xi=0$ the inequality follow from the assumption $a(x,0)=0$), we can bound the r.h.s. in the last math display from above by
\begin{align*}
 p \int \eta^{p-1} u^{\beta} \mu(x)^{\frac 1p} (a(x,\nabla u)\cdot \nabla u)^{\frac{p-1}{p}} |\nabla \eta| &=
 p \int u^{\beta-(\beta-1)\frac{p-1}{p}} \mu(x)^{\frac 1p} |\nabla \eta| (\eta^p u^{\beta -1} a(x,\nabla u)\cdot \nabla u)^{\frac{p-1}{p}} \\
 &\le 
 p \biggl( \int u^{p+\beta-1} \mu(x) |\nabla \eta|^p \biggr)^{\frac 1p} \biggl( \int \eta^p u^{\beta-1} a(x,\nabla u)\cdot \nabla u) \biggr)^{\frac{p-1}{p}},
\end{align*}
where in the second step we applied H\"older inequality with exponents $p$ and $\frac{p}{p-1}$, respectively. Observe that the last term on the r.h.s. appears on the l.h.s. in~\eqref{eq01}, so that after absorbing it we get from~\eqref{eq01} 
\begin{equation*}
 \beta \biggl( \int \eta^p u^{\beta -1} a(x,\nabla u) \cdot \nabla u\biggr)^{\frac 1p} \le 
 p \biggl( \int u^{p+\beta-1} \mu(x) |\nabla \eta|^p \biggr)^{\frac 1p},
\end{equation*}
which after taking the $p$-th power turns into
\begin{equation*}
 \int \eta^p u^{\beta -1} a(x,\nabla u) \cdot \nabla u  \le 
 \biggl( \frac p \beta \biggr)^{p} \int u^{p+\beta-1} \mu(x) |\nabla \eta|^p.
\end{equation*}
By definition of $\lambda$ in~\eqref{def1} in form of $\lambda(x)|\xi|^p \le a(x,\xi)\cdot \xi$ for any $\xi \in \R^d$, one has $\lambda(x) |\nabla u|^p \le a(x,\nabla u)\cdot \nabla u$, thus implying the claimed Caccioppoli  inequality \eqref{Caccioppoli}.
 
\PfStep{pf1}{pf1s2} Improvement of integrability. \\
We claim that there exists $c=c(d,p,s)\in[1,\infty)$ such that for $\frac 12 \le \rho < \sigma \le 1$ and $\alpha \ge 1$ it holds
\begin{equation}\label{eq04}
 \|\nabla(u^\alpha)\|_{L^{\frac{pt}{t+1}}(B_\rho)}
 \le 
  c (\sigma-\rho)^{-\frac{d}{d-1}} \Lambda(B_\sigma)^\frac{1}{p}
 \|u^\alpha\|_{W^{1,s_*}(B_\sigma \setminus B_\rho)}.
\end{equation}

Let $\eta \in \mathcal C^1_0(B_\sigma)$, $\eta \ge 0$, with $\eta = 1$ in $B_\rho$. First, we rewrite the Caccioppoli  inequality \eqref{Caccioppoli} from Step~\ref{pf1s1} as inequality for $u^{1 + \frac{\beta-1}{p}}$:
\begin{equation}\label{eq02}
  \left( \frac{p}{p+\beta-1} \right)^p \int \eta^p \lambda(x) |\nabla (u^{1+\frac{\beta-1}{p}})|^p \le 
  \biggl( \frac p \beta \biggr)^{p} \int \mu(x) (u^{1+\frac{\beta-1}{p}})^p |\nabla \eta|^p.
\end{equation}
Calling $v:=u^{1 + \frac{\beta-1}{p}}$, we can estimate the r.h.s. with the help of Lemma~\ref{lm1}, yielding
\begin{equation*}
  \int \eta^p \lambda(x) |\nabla v|^p
  \le c\left( \frac{p+\beta-1}{\beta} \right)^p (\sigma-\rho)^{-\frac{pd}{d-1}} \|\mu\|_{L^s(B_\sigma \setminus B_\rho)} \bigl( \|\nabla v\|^p_{L^{s_*}(B_\sigma \setminus B_\rho)} + \rho^{-p} \|v\|^p_{L^{s_*}(B_\sigma \setminus B_\rho)} \bigr). 
\end{equation*}
Using H\"older inequality with exponents $(\frac{t+1}{t},t+1)$ and the fact that $\eta = 1$ in $B_\rho$, we see that 
\begin{equation*}
 \|\nabla v\|_{L^{\frac{pt}{t+1}}(B_\rho)}^p 
 \le \|\lambda^{-1}\|_{L^t(B_\rho)} \|\lambda |\nabla v|^p \|_{L^1(B_\rho)}
 \le \|\lambda^{-1}\|_{L^t(B_\rho)} \int \eta^p \lambda(x) |\nabla v|^p.  
\end{equation*}
Using that $\frac 12 \le \rho \le \sigma \le 1$, combination of two previous relations yields
\begin{equation*}
 \|\nabla v\|_{L^{\frac{pt}{t+1}}(B_\rho)}^p 
 \le 
 c\left( \frac{p+\beta-1}{\beta} \right)^p (\sigma-\rho)^{-\frac{pd}{d-1}} \Lambda(B_\sigma) 
 \|v\|^p_{W^{1,s_*}(B_\sigma \setminus B_\rho)},
\end{equation*}
which after taking $p$-root turns into 
\begin{equation*}
 \|\nabla(u^\alpha)\|_{L^{\frac{pt}{t+1}}(B_\rho)}
 \le 
  c (\sigma-\rho)^{-\frac{d}{d-1}} \Lambda(B_\sigma)^\frac{1}{p}
 \|u^\alpha\|_{W^{1,s_*}(B_\sigma \setminus B_\rho)}, 
\end{equation*}
with $\alpha := 1+\frac{\beta-1}{p}$. 

\PfStep{pf1}{pf1s3} One-step Improvement.

First, we note that \eqref{pqcond} and $t>\frac1{p-1}$ imply $\delta:= \frac{1}{s_*} - \frac1p(1+\frac{1}{t})>0$. In particular it holds $s_*<\frac{t p}{t+1}$. We claim that there exists $c=c(d,s,t,p)$ such that for $\frac 12 \le \rho < \sigma \le 1$ there holds
\begin{equation}\label{eq05}
 \|u^{\chi\alpha}\|_{W^{1,s_*}(B_\rho)}^{\frac{1}{\chi \alpha}} \le \biggl( \frac{ c\Lambda(B_\sigma)^\frac{1}{p} }{(\sigma-\rho)^{\frac{d}{d-1}} } \biggr)^{\frac{1}{\chi \alpha}} 
 \|u^\alpha\|_{W^{1,s_*}(B_\sigma)}^\frac{1}{\alpha},
\end{equation}
where $\chi:=1+\delta>1$. Using H\"older inequality with exponent $\frac{pt}{(t+1)s_*}>1$ and its dual exponent $\frac{pt}{pt-(t+1)s_*}=\frac1{\delta s_*}$ we get 
\begin{align*}
 \biggl( \int_{B_\rho} |\nabla(u^{(1+\delta)\alpha})|^{s_*} \biggr)^{\frac1{s_*}} 
 &= (1+\delta)\alpha \biggl( \int_{B_\rho} |\nabla u|^{s_*} u^{(\alpha-1)s_*} u^{\alpha \delta s_*} \biggr)^{\frac1{s_*}}
 = (1+\delta) \biggl( \int_{B_\rho} |\nabla(u^{\alpha})|^{s_*} u^{\alpha \delta s_*} \biggr)^{\frac1{s_*}}  
 \\
 &\le (1+\delta) \biggl( \int_{B_\rho} |\nabla(u^\alpha)|^{\frac{pt}{t+1}} \biggr)^{\frac{t+1}{pt}} 
 \biggl( \int_{B_\rho} u^\alpha \biggr)^{\delta}.
\end{align*}
Combining the above estimate with \eqref{eq04} from Step~\ref{pf1s2}, we get (recall $\chi = 1+\delta$) 
\begin{equation*}
 \|\nabla(u^{\chi\alpha})\|_{L^{s_*}(B_\rho)} \le c (\sigma-\rho)^{-\frac{d}{d-1}} \Lambda(B_\sigma)^\frac{1}{p}
 \|u^\alpha\|_{W^{1,s_*}(B_\sigma)}^\chi,
\end{equation*}
where we hided $\chi = 1 + \delta < \frac{d}{d-1}$ into $c$. In order to have full $W^{1,s_*}(B_\rho)$-norm also on the l.h.s., using $s_* \ge 1$ as well as $\chi < \frac{d}{d-1}$ we can use Sobolev inequality to the effect
\begin{equation*}
 \| u^{\chi\alpha} \|_{L^{s_*}(B_\rho)} \le c \| u^{\alpha} \|_{W^{1,s_*}(B_\rho)},
\end{equation*}
thus obtaining the claim.

\PfStep{pf1}{pf1s4} Iteration.\\
We iterate the outcome of Step~\ref{pf1s3}. For $\bar \alpha \ge 1$ and $n \in \mathbb{N}$ let $\alpha_n := \bar \alpha \chi^{n-1}$, $\rho_n := \frac 12 + \frac{1}{2^n+1}$, $\sigma_n := \rho_n + \frac{1}{2^{n+1}} = \rho_{n-1}$. Then~\eqref{eq05} from Step~\ref{pf1s4} with $\alpha := \alpha_n$ has the form 
\begin{equation}\label{eq05}
 \|u^{\alpha_{n+1}}\|_{W^{1,s_*}(B_{\rho_n})}^{\frac{1}{\alpha_{n+1}}} \le ( c\Lambda(B_1)^\frac{1}{p}4^n)^{\frac{1}{\bar\alpha \chi^n}} 
 \|u^{\alpha_n}\|_{W^{1,s_*}(B_{\rho_{n-1}})}^\frac{1}{\alpha_n}.
\end{equation}
Using that $L^p$ approximates $L^\infty$ as $p \to \infty$, we see that 
\begin{align}\label{est:sup:general}
 \|u\|_{L^\infty(B_{1/2})} 
 &\le 
 \biggl( \prod_{n=1}^\infty ( c\Lambda(B_\sigma)^\frac{1}{p}4^n)^{\frac{1}{\bar\alpha \chi^n}} \biggr) 
 \|u^{\bar\alpha}\|_{W^{1,s_*}(B_1)}^\frac{1}{\bar\alpha}\notag
 \\
 &\le 
 c\Lambda(B_\sigma)^{\frac{1}{p\bar \alpha}\frac{1}{\chi-1}} 
 \|u^{\bar\alpha}\|_{W^{1,s_*}(B_1)}^\frac{1}{\bar\alpha},
\end{align}
which for $\bar \alpha = 1$ yields the desired claim where we use $\chi=1+\delta$ and $s_*\leq \frac{tp}{t+1}$. 

\PfStep{pf1}{pf1s5} The remaining case $1+\frac1t<\frac{p}{d-1}$.

Using Fubini theorem, we can choose a generic radius $r_0 \in (\frac 12,1)$ such that 
 \begin{equation*}
    \|u_+\|_{W^{1,\frac{pt}{t+1}}(S_{r_0})}^\frac{pt}{t+1} \le 2\|u_+\|_{W^{1,\frac{pt}{t+1}}(B_1)}^\frac{pt}{t+1}. 
 \end{equation*}
We test the weak formulation of $-\nabla \cdot a(x,\nabla u)\leq0$ see \eqref{def:harmonic} with the non-negative test function $\phi := (u_+-\sup_{S_{r_0}} u_+)_+$, which obviously vanishes on $S_{r_0}$ and can be therefore trivially extended by zero to the whole domain $\Omega$. This yields
\begin{equation*}
 0 \stackrel{\eqref{def:harmonic}}\ge \int_{B_{r_0}} a(x,\nabla u)\cdot \nabla \phi = \int_{B_{r_0}} a(x,\nabla \phi)\cdot \nabla \phi \stackrel{\eqref{def1}}\ge \int_{B_{r_0}} \lambda(x) |\nabla \phi|^p.
\end{equation*}
In particular, we see that $\nabla \phi = 0$ a.e. in $B_{r_0}$, hence $\phi \equiv 0$ and thus 
\begin{equation*}
 \|u_+\|_{L^\infty(B_{\frac 12})} \le \|u_+\|_{L^\infty(B_{r_0})} \le \sup_{S_{r_0}} u_+. 
\end{equation*}

Using that $\frac{pt}{t+1} > d-1$, which follows from $1+\frac 1t < \frac p{d-1}$, we have by Sobolev embedding that $\sup_{S_{r_0}} u_+ \le c\|u_+\|_{W^{1,\frac{pt}{t+1}}(S_{r_0})}$ for some $c=c(d,p,t)>0$ which by the above choice of $r_0$ completes the claim.

\end{proof}

\begin{proof}[Proof of Proposition~\ref{P:2d}]
This follows exactly as in Step~5 of the proof of Theorem~\ref{T:est} using that for $d=2$ it holds $\sup_{S_{r_0}} u_+ \le c\|u_+\|_{W^{1,1}(S_{r_0})}$.
\end{proof}

We close this section by deriving from Theoem~\ref{T:est} in the case $s>1$ an $L^\infty-L^\gamma$ estimate . 

\begin{corollary}\label{C:est}
 Let $d \ge 2$, $\Omega\subset \R^d$ and $p\in(1,\infty)$. Moreover, let $s\in(1,\infty]$ and $t\in(\frac1{p-1},\infty]$ satisfy \eqref{pqcond}. Let $a : \Omega \times \R^d \to \R^d$ be a Caratheodory function with $a(\cdot,0)\equiv0$ such that $\lambda$ and $\mu$ defined in~\eqref{def1} satisfy $\mu \in L^s(\Omega)$ and $\frac 1\lambda \in L^t(\Omega)$. Then, any weak subsolution $u$ of \eqref{eq} and any $\gamma>0$ there exists $c=c(\gamma,d,p,s,t) \in [1,\infty)$ such that for any ball $B_{R} \subset \Omega$ 
 \begin{equation*}
  \sup_{B_{R/2}} u 
  \le 
 c \Lambda(B_R)^{\frac1\gamma\frac{s}{s-1}(1+\frac1\delta)}    \biggl(\fint_{B_R}u_+^\gamma\biggr)^\frac{1}{\gamma}.   
 \end{equation*}

\end{corollary}

\begin{proof}\PfStart{pf2}

Without loss of generality we consider $R=1$ and suppose that $B_1\Subset\Omega$. Caccioppoli inequality \eqref{eq02} with $\beta=1+p(\alpha-1)$ for $\alpha\geq1$ and $\eta\in C_c^1(B_1)$ with $\eta=1$ on $B_\frac12$ and $|\nabla \eta|\leq 2$ and H\"older inequality yield
\begin{align*}
\| \nabla (u_+^\alpha)\|_{L^\frac{pt}{t+1}(B_{1/2})}^p\leq&\|\lambda^{-1}\|_{L^t(B_1)}\int_{B_1}\eta^p\lambda|\nabla (u_+^\alpha)|^p\leq (2p)^p\|\lambda^{-1}\|_{L^t(B_1)}\int_{B_1}\mu u_+^{\alpha p}\\
\leq& (2p)^p\|\lambda^{-1}\|_{L^t(B_1)}\|\mu\|_{L^s(B_1)}\|u_+^\alpha\|_{L^{\frac{s}{s-1}p}(B_1)}.
\end{align*}
The above inequality combined with $\frac{tp}{t+1}\leq p\leq \frac{sp}{s-1}$ implies $\|u_+^\alpha\|_{W^{1,\frac{tp}{t+1}}(B_{1/2})}^\frac1\alpha \le c \Lambda(B_1)^\frac{1}{\alpha p} \| u_+ \|_{L^{\frac{\alpha ps}{s-1}}(B_1)}$ (note that $1\leq\Lambda(B_r)$) for some $c=c(d,p)\in[1,\infty)$. Hence, we have in combination with \eqref{est:sup:general} that
%
\begin{equation}\label{eq06}
 \|u_+\|_{L^\infty(B_{1/4})} \le 
 c \Lambda(B_1)^{\frac{1}{ \alpha p} (1+\frac1{\delta})} \|u_+\|_{L^{\frac{\alpha ps}{s-1}}(B_1)},
\end{equation}
where $c=c(\alpha, d,p,t,s)\in[1,\infty)$.

From estimate \eqref{eq06} the claim follows by routine arguments and we only sketch the idea (see \cite[Proof of Theorem 3.3, Step 2]{BS19a} for precise arguments in the case $p=2$). By scaling and translation, we deduce from \eqref{eq06} that for all $\rho>0$ and $x\in B_1$ such that $B_\rho(x)\subset B_1$ it holds for $\alpha\geq1$
\begin{equation*}
 \|u_+\|_{L^\infty(B_{\rho/4}(x))} \le 
 c \Lambda(B_\rho(x))^{\frac{1}{\alpha p} (1+\frac1{\delta})} \rho^{-\frac{d}p(1-\frac1s)}\|u_+\|_{L^{\frac{\alpha ps}{s-1}}(B_\rho(x))},
\end{equation*}
where $c$ is as in \eqref{eq06}. Combining the above estimate with a simple covering argument, we obtain that there exists $c=c(\alpha,d,p,s,t)\in[1,\infty)$ such that for all $\theta\in(0,1)$ and $r\in(0,1]$ it holds
$$
\|u_+\|_{L^\infty(B_{\theta r})} \le  c \Lambda(B_r)^{\frac{1}{\alpha p} (1+\frac1{\delta})} (1-\theta)^{-\kappa}r^{-d\frac{s-1}{\alpha ps}}\|u_+\|_{L^{\frac{\alpha ps}{s-1}}(B_r)},
$$
where $\kappa:=\frac{d}{\alpha p}((\frac1t+\frac1s)(1+\frac1{\delta})+1-\frac1s)$ which is the claim for all $\gamma \geq\frac{ps}{s-1} $ (by choosing $\alpha=\frac{s-1}{ps}\gamma$). The claim for $\gamma\in(0,\frac{ps}{s-1})$ follows by a standard interpolation and iteration argument see e.g.\ the textbook reference \cite[p.\ 75]{HL} in the uniformly elliptic case or as mentioned above \cite[Proof of Theorem 3.3, Step 2]{BS19a} for a closely related setting.
 \end{proof}

\section{Counterexample, proof of Theorem~\ref{T:negative}}

\begin{proof}[Proof of Theorem~\ref{T:negative}]\PfStart{pf3} 
The following construction is very much inspired by a construction in \cite{FSS98} in the linear case, that is $p=2$, and $d=4$ (which was already extended to $d\geq3$ in \cite{Schwarzmannthesis}). 

Let $d\geq3$. Throughout the proof, we set
$$
x=(x_1,\dots,x_d)=(x_1,x')\quad\mbox{and}\quad |x'|=\sqrt{\sum_{j=2}^dx_j^2}.
$$
For any $p\in(1,\infty)$ and $\theta\in[0,1]$, we define $\lambda_\theta(x):=\omega_\theta(|x'|)$ where $\omega_\theta:(0,1)\to\R_+$ is defined as
\begin{equation}\label{def:omegatheta}
\omega_\theta(r)=\begin{cases}
(i+1)^{(p-1)\theta} 4^{-pi\theta}&\mbox{when $r\in[\frac12 4^{-i},4^{-i})$},\\
((i+1)^{-(p-1)}4^{pi})^{1-\theta}&\mbox{when $r\in[\frac14 4^{-i},\frac12 4^{-i})$}
\end{cases}
\end{equation}
for $i\in\mathbb N$. We will construct an explicit subsolution to $-\nabla \cdot (\lambda_\theta |\nabla v|^{p-2}\nabla v)=0$, which is of the form
\begin{equation}\label{def:v}
v(x)=e^{\alpha x_1}\phi(|x'|)
\end{equation}
for some parameter $\alpha=\alpha(d,p)>0$ and $\phi:(0,1)\to\R$ is defined by
\begin{equation}\label{def:phi}
\phi(r)=\begin{cases}
i+\frac{\eta_i}{2^Q-1}((4^ir)^{-Q}-1)&\mbox{when $r\in[\frac12 4^{-i},4^{-i})$},\\
(i+1)-(1-\eta_i)(4^{i+1}r-1)^2&\mbox{when $r\in[\frac14 4^{-i},\frac12 4^{-i})$}
\end{cases},\quad\mbox{with}\quad Q=\begin{cases}\max\{d-3,1\}&\mbox{if $p\geq2$}\\ \frac{d-2}{p-1}-1&\mbox{if $1<p<2$}\end{cases}
\end{equation}
where $\eta_i\in[0,1]$ will be specified below. Note that $Q>0$ and $\phi$ is continuous by definition. We choose $\eta_i\in(0,1)$ such that the flux $\lambda_\theta|\nabla v|^{p-2}\nabla v$ is continuous at $|x'|=\frac12 4^{-i}$ for every $i\in\mathbb N$. More precisely, we set $\eta_i$ to be the largest constant (in $[0,1]$) satisfying
\begin{equation}\label{def:etai}
F_i(\eta_i)=0,
\end{equation}
where $F_i:(0,1]\to\R$ is given by
\begin{align*}
F_i(\eta):=&\sqrt{(\alpha (i+\eta)4^{-i})^2+(C_Q\eta)^2}^{p-2}C_Q\eta \\
&- \sqrt{(\alpha(i+\eta)(i+1)^{-1})^2+(8(1-\eta)4^{i}(i+1)^{-1})^2}^{p-2} 8(1-\eta)4^{2i}(i+1)^{-1}
\end{align*}
with
\begin{equation*}
C_Q=Q\frac{2^{Q+1}}{2^Q-1}.
\end{equation*}
Note that $\eta_i$ is well-defined since $F_i:(0,1)\to\R$ is continuous with
$$
\lim_{\eta\to0}F_i(\eta)=- \sqrt{(\alpha i)^2+(2\cdot 4^{i+1})^2}^{p-2} 8\cdot4^{2i}(i+1)^{-(p-1)}<0
$$
and
$$
\lim_{\eta\to1}F_i(\eta)=\sqrt{(\alpha (i+1)4^{-i})^2+C_Q^2}^{p-2}C_Q>0.
$$
The definition of $\eta_i$ is rather implicit and we provide now some explicit bounds on $\eta_i$ which will be useful for later computations. We distinguish two cases. For $p\geq2$ and $\alpha\geq C_Q$, we have that
\begin{equation}\label{eta:loweboundgeq}
\exists j=j(d,p)\geq2\mbox{ such that $\forall i\geq j$:}\quad\eta_i \geq 1-8^{-1} (4^{p-2} C_Q) 4^{-2i}(i+1)=:\underline \eta_i.
\end{equation} 
Indeed, let $j=j(d,p)\geq2$ be such that $\underline \eta_i\in(0,1)$ for all $i\geq j$. By definition of $\eta_i$, it suffices to show that $F_i(\underline \eta_i)\leq0$ for $i\geq j$. We have
\begin{align*}
F_i(\underline \eta_i)\leq&\sqrt{(\alpha (i+1)4^{-i})^2+C_Q^2}^{p-2}C_Q - \sqrt{(\alpha i/(i+1))^2}^{p-2} (4^{p-2} C_Q)\\
=&\sqrt{((i+1)4^{-i})^2+(C_Q/\alpha)^2}^{p-2}\alpha^{p-2}C_Q - \alpha^{p-2}\sqrt{(i/(i+1))^2}^{p-2} (4^{p-2} C_Q)\\
\leq&\alpha^{p-2}(2^{p-2}C_Q-2^{-(p-2)}(4^{p-2} C_Q))=0,
\end{align*}
where we used for the last inequality $(i+1)4^{-i}\leq 1$ and $i/(i+1)\geq\frac12$ for $i\geq1$ and $\alpha\geq C_Q$. 

In the case $p\in(1,2)$, we have for $\alpha\geq 2^\frac{2-p}{p-1}C_Q$ that
\begin{equation}\label{eta:loweboundleq}
\exists j=j(\alpha,d,p)\geq2\mbox{ such that $\forall i\geq j$:}\quad\eta_i \geq 1-8^{-1} \alpha 4^{-2i}(i+1)=:\overline \eta_i.
\end{equation} 
Indeed, this follows as above from
\begin{align*}
F_i(\overline \eta_i)\leq&C_Q^{p-1} - \sqrt{\alpha^2+(\alpha 4^{-i})^2}^{p-2} \alpha\leq C_Q^{p-1}-\alpha^{p-1}2^{p-2}\leq0.
\end{align*}
%


\PfStep{pf3}{pf3s1} We show that for every $\alpha\geq\max\{1,2^{\frac{2-p}{p-1}}\}C_Q$, the function $v$ defined in \eqref{def:v} has finite energy, that is $\int_{B_1}\lambda_\theta (|v|^p+|\nabla v|^p)<\infty$ provided $(1-\theta)p<d-1$. 

We show first $\int_{B_1}\lambda_\theta |v|^p<\infty$. For this, we observe that $0\leq\phi(r)\leq\log(4/r)$ for all $r\in(0,1)$. Indeed, $\phi\geq0$ is clear from the definition \eqref{def:phi} and for $r\in[\frac14 4^{-i},4^{-i})$, we have 
$$\phi(r)\leq i+1=\log_4(4^{i+1})\leq \log_4(\tfrac{4}r)\leq \log(\tfrac{4}r).$$
Similarly, we get
\begin{equation}\label{estabove:omegatheta}
\omega_\theta(r)\leq \begin{cases}
((2r)^p\log(4/r)^{p-1})^\theta&\mbox{when $r\in[\frac12 4^{-i},4^{-i})$},\\
(r^p\log(2/r)^{p-1})^{-(1-\theta)}&\mbox{when $r\in[\frac14 4^{-i},\frac12 4^{-i})$}
\end{cases}.
\end{equation}
Hence, there exists $C=C(\alpha,d,p)>0$ such that
\begin{align*}
\int_{B_1}\lambda_\theta v^p\,dx\leq C\int_0^1r^{-(1-\theta)p}\log(2/r)^{p-(1-\theta)(p-1)}r^{d-2}\,dr<\infty,
\end{align*}
where the last integral is finite since $(1-\theta)p< d-1$.

Next, we show $\int_{B_1}\lambda_\theta |\nabla v|^p<\infty$. For this we compute the gradient of $v$:
\begin{equation}\label{normnablav}
\nabla v=\begin{pmatrix}\alpha \phi \\ \phi' \frac{x'}{|x'|}\end{pmatrix}e^{\alpha x_1}\quad\mbox{and}\quad|\nabla v|=\sqrt{\alpha^2 \phi^2+\phi'^2}e^{\alpha x_1}.
\end{equation}
Moreover, we compute
\begin{equation}\label{phiprime}
\phi'(r)=\begin{cases}
-Q\frac{\eta_i}{2^Q-1}(4^ir)^{-Q}r^{-1}&\mbox{when $r\in(\frac12 4^{-i},4^{-i})$},\\
-2(1-\eta_i)4^{i+1}(4^{i+1}r-1)&\mbox{when $r\in(\frac14 4^{-i},\frac12 4^{-i})$}
\end{cases}
\end{equation}
and for later usage
\begin{equation}\label{phiprime2}
\phi''(r)=\begin{cases}
Q(Q+1)\frac{\eta_i}{2^Q-1}(4^ir)^{-Q}r^{-2}&\mbox{when $r\in(\frac12 4^{-i},4^{-i})$},\\
-2(1-\eta_i)4^{2(i+1)}&\mbox{when $r\in(\frac14 4^{-i},\frac12 4^{-i})$}
\end{cases}.
\end{equation}
From \eqref{eta:loweboundgeq} and \eqref{eta:loweboundleq}, we obtain that there exists $C=C(\alpha,d,p)>0$ such that $0\leq 1-\eta_i\leq C 4^{-2i}(i+1)$ for $i\geq j(\alpha,d,p)$ and thus in combination with \eqref{phiprime}  there exists $C=C(\alpha,d,p)>0$ such that
\begin{equation*}
 |\phi'(r)|\leq C\begin{cases}
r^{-1}&\mbox{when $r\in(\frac12 4^{-i},4^{-i})$},\\
\log(2/r)r&\mbox{when $r\in(\frac14 4^{-i},\frac12 4^{-i})$}
\end{cases}
\end{equation*}
for all $i\geq j$. Hence, we find $C=C(\alpha,d,p)>0$ such that
\begin{align*}
\int_{B_1}\lambda_\theta |\nabla v|^p\leq C+C\int_0^1 \biggl((r^p\log(2/r)^{p-1})^\theta r^{-p}+(r^p\log(2/r)^{p-1})^{-(1-\theta)}(\log(2/r)r)^p\biggr)r^{d-2}\,dr<\infty,
\end{align*}
where we use again $(1-\theta)p<d-1$. Finally, it is easy to check that the sequence $(v_k)_k$ defined by $v_k(x)=e^{\alpha x_1}\phi_k(|x'|)$ with $
\phi_k(x)=\phi(x)$ if $|x|>4^{-k}$ and $\phi_k(x)=k$ if $|x'|\leq 4^{-k}$ is a sequence of Lipschitz functions satisfying $\lim_{k\to\infty}\int_{B_1}\lambda_\theta(|v-v_k|^p+|\nabla v-\nabla v|^p)\to0$ as $k\to\infty$ and a straightforward regularization shows that $v$ in $H^{1,p}(B_1,a)$ with $a(x,\xi):=\lambda_\theta(x)|\xi|^{p-2}\xi$.

\PfStep{pf3}{pf3s2} We claim that  there exist $\alpha_0=\alpha_0(d,p)\geq1$ such that for every $\alpha\geq\alpha_0$ there exists $\rho=\rho(\alpha,d,p)\in(0,1]$ such that $v$ defined in \eqref{def:v} is a weak subsolution in $\{x\in B_1\,:\,\delta<|x'|<\rho\}$ for all $\delta>0$. 

For this, we observe first that by \eqref{normnablav} the nonlinear strain $|\nabla v|^{p-2}\nabla v$ of $v$ is given by
\begin{equation}\label{normnablav1}
 |\nabla v|^{p-2}\nabla v=\sqrt{\alpha^2\phi^2+\phi'^2}^{p-2}\begin{pmatrix}\alpha \phi \\ \phi' \frac{x'}{|x'|}\end{pmatrix}e^{\alpha (p-1) x_1}.
\end{equation}

Introducing the notation $M_{2i}=B_1\cap \{\frac12 4^{-i}<|x'|<4^{-i}\}$ and $M_{2i+1}=B_1\cap \{\frac14 4^{-i}<|x'|<\frac124^{-i}\}$, we obtain with help of integrating by parts
\begin{align*}
\int_{B_1} \lambda_\theta|\nabla v|^{p-2}\nabla v\cdot \nabla\varphi=&\sum_{i\in\mathbb N}\int_{M_i}\omega_\theta|\nabla v|^{p-2}\nabla v\cdot \nabla\varphi\\
=&\sum_{i\in\mathbb N}-\int_{M_i}\omega_\theta\nabla \cdot (|\nabla v|^{p-2}\nabla v)\varphi+\int_{\partial M_i}\omega_\theta |\nabla v|^{p-2}\nabla v\cdot \nu \varphi\\
=&\sum_{i\in\mathbb N}-\int_{M_i}\omega_\theta\nabla \cdot (|\nabla v|^{p-2}\nabla v)\varphi+\int_{\partial M_i}\omega_\theta\sqrt{\alpha^2 \phi^2+\phi'^2}^{p-2}\phi' e^{(p-1)\alpha x_1}\varphi,
\end{align*}
where $\nu$ denotes the outer unit normal to $M_i$ that is $\nu=(0,x'/|x'|)$. Hence, it suffices to show that there exists $\alpha_0>0$ such that for all $\alpha\geq\alpha_0$ there exists $j=j(\alpha,d,p)\geq2$ such that
\begin{enumerate}[(i)]
\item $v$ satisfies $\nabla \cdot (|\nabla v|^{p-2}\nabla v)\geq0$ in the classical sense in each shell $M_i$ for all $i\geq j$;
\item the flux has only nonnegative jumps at the interfaces, that is
\begin{align*}
(\omega_\theta\sqrt{\alpha^2 \phi^2+\phi'^2}^{p-2}\phi')(\gamma_-):=&\lim_{r\to \gamma\atop r<\gamma}(\omega_\theta \sqrt{\alpha^2 \phi^2+\phi'^2}^{p-2}\phi')(r)\\
\leq&\lim_{r\to \gamma\atop r>\gamma}(\omega_\theta \sqrt{\alpha^2 \phi^2+\phi'^2}^{p-2}\phi')(r)=:(\omega_\theta |\nabla v|^{p-2}\phi')(\gamma_+)
\end{align*}
for all $\gamma\in\bigcup_{i\in\mathbb N, i\geq j}\{4^{-i}\}\cup\{\frac12 4^{-i}\}$.
\end{enumerate}

\substep{2.1} Argument for (i). Let $\alpha\geq1$ be such that
\begin{equation}\label{ass:alpha}
\alpha\geq \alpha_0(p,d):=\max\biggl\{1,C_Q,2^\frac{2-p}{p-1}C_Q,2^p\sqrt{C_Q(1+\frac{d-2}{p-1})},8\frac{d-1}{p-1}\biggr\}
\end{equation}
%
and let $j=j(\alpha,d,p)\geq2$ be such that the estimates \eqref{eta:loweboundgeq} and \eqref{eta:loweboundleq} are valid.

We show that $v$ with $\alpha$ as above, satisfies $\nabla \cdot (|\nabla v|^{p-2}\nabla v)\geq0$ in the classical sense in each shell $M_i$ for all $i\geq j$. We compute with help of \eqref{normnablav1} on $M_i$
\begin{align}\label{eq:plaplacev}
&\nabla \cdot (|\nabla v|^{p-2}\nabla v)\notag\\
=&\biggl(\alpha^2(p-1)\sqrt{\alpha^2\phi^2+\phi'^2}^{p-2}\phi+(p-2)\sqrt{\alpha^2\phi^2+\phi'^2}^{p-4}|\phi'|^2(\alpha^2\phi+\phi'')\notag\\
&\quad+\sqrt{\alpha^2\phi^2+\phi'^2}^{p-2}(\phi''+(d-2)\frac{\phi'}{|x'|})\biggr)e^{\alpha(p-1) x_1}\notag\\
=&\sqrt{\alpha^2\phi^2+\phi'^2}^{p-4}\biggl(\alpha^2(p-1)(\alpha^2 \phi^2+\phi'^2)\phi+(p-2)\phi'^2(\alpha^2\phi+\phi'')\notag\\
&\qquad\qquad\qquad+(\alpha^2 \phi^2+\phi'^2)(\phi''+(d-2)\frac{\phi'}{|x'|})\biggr)e^{\alpha(p-1) x_1}.
\end{align}
We show that $v$ is a classical subsolution in  $M_{2i+1}$. Note that $\phi>0$ and $\phi',\phi''<0$ on $(\frac14 4^{-i},\frac12 4^{-i})$. 

We consider first the case $p\geq2$. From $\phi>0$, $\phi''<0$ and $\phi'^2\leq \alpha^2\phi^2+\phi'^2$, we deduce
$$
(p-2)\phi'^2(\alpha^2\phi+\phi'')\geq (p-2)(\alpha^2\phi^2+\phi'^2)\phi''
$$
and in combination with \eqref{eq:plaplacev} that
\begin{align*}
\nabla \cdot (|\nabla v|^{p-2}\nabla v)\geq&\sqrt{\alpha^2\phi^2+\phi'^2}^{p-2}(\alpha^2(p-1)\phi+(p-1)\phi''+(d-2)\frac{\phi'}{|x'|})e^{\alpha(p-1) x_1}.
\end{align*}  
Hence, $\nabla \cdot (|\nabla v|^{p-2}\nabla v)\geq0$ on $M_{2i+1}$ is equivalent to
\begin{align*}
\alpha^2(p-1)\phi(r)+(p-1)\phi''(r)+(d-2)\frac{\phi'(r)}{r}\geq0\quad \mbox{for all $r\in(\frac14 4^{-i},\frac12 4^{-i})$},
\end{align*}
which is by \eqref{def:phi}, \eqref{phiprime} and \eqref{phiprime2} valid if and only if
$$
\alpha^2(p-1)\left(i+1-(1-\eta_i)(4^{i+1}r-1)^2\right)-2(1-\eta_i)4^{i+1}((p-1)4^{i+1}+r^{-1}(d-2)(4^{i+1}r-1))\geq 0
$$
for all $r\in(\frac14 4^{-i},\frac12 4^{-i})$. We estimate with help of $\eta_i\in[0,1]$,
\begin{align*}
& \alpha^2(p-1)\left((i+1)-(1-\eta_i)(4^{i+1}r-1)^2\right)-2(1-\eta_i)4^{i+1}((p-1)4^{i+1}+r^{-1}(d-2)(4^{i+1}r-1))\\
\geq&\alpha^2(p-1)i-2(1-\eta_i)4^{2(i+1)}(p-1+d-2).
\end{align*}
The lower bound on $\eta_i\geq\underline \eta_i$, see \eqref{eta:loweboundgeq}, implies $1-\eta_i\leq 1-\underline \eta_i\leq 8^{-1} (4^{p-2} C_Q) 4^{-2i}(i+1)$ and thus
\begin{align}\label{subsolution:quadraticphase:1}
\alpha^2(p-1)i-2(1-\eta_i)4^{2(i+1)}(p-1+d-2)\geq&\alpha^2(p-1)i-4^{p-1} C_Q(i+1) (p+d-3)\geq0,
\end{align}
where the last inequality is valid since $(i+1)/i\leq2$ for $i\geq1$ and $\alpha^2\geq4^{p-1}C_Q2(1+\frac{d-2}{p-1})$ (which is ensured by $\alpha\geq\alpha_0$, see \eqref{ass:alpha}).

\medskip

Next, we consider the case $p\in(1,2)$. We deduce from \eqref{eq:plaplacev} with $p-2<0$ and $\phi>0$, $\phi',\phi''<0$ that
\begin{align}\label{subsolution:quadraticphasepleq2}
&\nabla \cdot (|\nabla v|^{p-2}\nabla v)\notag\\
\geq&\sqrt{\alpha^2\phi^2+\phi'^2}^{p-4}\left((\alpha^2 \phi^2+\phi'^2)\left(\alpha^2(p-1)\phi+\phi''+(d-2)\frac{\phi'}{|x'|}\right)-(2-p)\phi'^2\phi\right)e^{\alpha(p-1) x_1}.
\end{align} 
Similar computations as above yield for all $r\in(\frac14 4^{-i},\frac12 4^{-i})$ and $p\in(1,2)$
\begin{eqnarray*}
\alpha^2(p-1)\phi(r)+\phi''(r)+(d-2)\frac{\phi'(r)}{r}&\geq& \alpha^2(p-1)(i+\eta_i)-2(1-\eta_i)4^{2(i+1)}(d-1)\\
&\stackrel{\eqref{eta:loweboundleq}}{\geq}&\alpha^2(p-1)i-4\alpha(i+1)(d-1)\geq1
\end{eqnarray*}
where the last inequality is valid for all $i\geq1$ and $\alpha\geq 8\frac{d-1}{p-1}$ (see \eqref{ass:alpha}). Inserting this into \eqref{subsolution:quadraticphasepleq2}, we obtain (using $2-p\leq1$)
\begin{eqnarray*}
\nabla \cdot (|\nabla v|^{p-2}\nabla v)&\geq&\sqrt{\alpha^2\phi^2+\phi'^2}^{p-4}\left(\alpha^2 \phi^2-\phi'^2\phi\right)e^{\alpha(p-1) x_1}\\
&\stackrel{\eqref{phiprime},\eqref{eta:loweboundleq}}\geq&\sqrt{\alpha^2\phi^2+\phi'^2}^{p-4}\phi\left(\alpha^2 \phi- (\alpha(i+1)4^{-i})^2\right)e^{\alpha(p-1) x_1}
\geq0,
\end{eqnarray*} 
where we use in the last inequality that $4^{-2i}(i+1)^2\leq1$ and $\phi\geq1$ on $(\frac14 4^{-i},\frac12 4^{-i})$ with $i\geq1$.
%
%
 
\bigskip

Now, we show that $v$ is a classical subsolution in  $M_{2i}$. In view of \eqref{eq:plaplacev} it suffices to show that for all $r\in(\frac12 4^{-i},4^{-i})$ it holds
\begin{equation}\label{need:logphase}
\alpha^4(p-1)\phi^3(r)+\alpha^2(2p-3)\phi(r)\phi'^2(r)+\phi'^2((p-1)\phi''(r)+\frac{d-2}r\phi'(r))+\alpha^2\phi^2(r)(\phi''(r)+\frac{d-2}r\phi'(r))\geq0
\end{equation}
%
For $p\geq\frac32$, we obviously have
$$
\alpha^4(p-1)\phi^3(r)+\alpha^2(2p-3)\phi(r)\phi'^2(r)\geq0\qquad\mbox{for all $r\in(\frac12 4^{-i},4^{-i})$}.
$$
Let us first consider $p\geq2$. In the case $d\geq4$, the choice of $\phi$ ensures
$$ 
\forall r\in(\frac12 4^{-i},4^{-i}):\quad\phi''(r)+\frac{d-2}r\phi'(r)=0\quad\mbox{and}\quad (p-1)\phi''(r)+\frac{d-2}r\phi'(r)=(p-2)\phi''(r)\geq0 
$$
and similarly for $d=3$ that $\phi''(r)+\frac{d-2}r\phi'(r)=\frac12 \phi''(r)\geq0$ and $(p-1)\phi''(r)+\frac{d-2}r\phi'(r)\geq0$. Altogether, we have that \eqref{need:logphase} is valid for all $r\in(\frac12 4^{-i},4^{-i})$ provided $p\geq2$.

Next, we consider the case $p\in(1,2)$. The choice of $\phi$ ensures
$$ 
\forall r\in(\frac12 4^{-i},4^{-i}):\quad(p-1)\phi''(r)+\frac{d-2}r\phi'(r)=0\quad\mbox{and}\quad \phi''(r)+\frac{d-2}r\phi'(r)=(2-p)\phi''(r)\geq0. 
$$
Using the above two identities, we see that \eqref{need:logphase} is equivalent to
$$
\alpha^4(p-1)\phi^3(r)+\alpha^2(2p-3)\phi(r)\phi'^2(r)+\alpha^2\phi^2(r)(2-p)\phi''(r)\geq0
$$
and thus it suffices to show 
$$
\alpha^2(2p-3)\phi\phi'^2+\alpha^2\phi^2(2-p)\phi''\geq0.
$$
For $p\in[\frac{3}2,2]$ the above inequality directly follows from $\phi,\phi''\geq0$ and it is left to consider $p\in(1,\frac{3}2)$ in which case the above inequality is equivalent to
$$
\frac{3-2p}{2-p}\frac{\phi'^2}{\phi''}\leq \phi.
$$
The above inequality is valid on $(\frac12 4^{-i},4^{-i})$ provided $i\geq 2$. Indeed, this follows from $\phi\geq i$ on $(\frac12 4^{-i},4^{-i})$ and
$$
\frac{3-2p}{2-p}\frac{\phi'^2}{\phi''}\leq \frac{3-2p}{2-p}\frac{Q}{Q+1}\frac{\eta_i}{2^Q-1}2^Q\leq 2\frac{Q}{Q+1}\leq 2.
$$

\substep{2.2} Argument for (ii). Let $\alpha\geq1$ and $j=j(\alpha,d,p)\geq2$ be as in Substep~2.1. 

In view of \eqref{normnablav}, we need to show that for all $\gamma\in\bigcup_{i\in\mathbb N,i\geq j}\{4^{-i}\}\cup\{\frac12 4^{-i}\}$ it holds
\begin{equation}\label{ineq:jumpcondition}
(\omega_\theta \sqrt{\alpha^2 \phi^2+\phi'^2}^{p-2}\phi')(\gamma_+)\geq(\omega_\theta \sqrt{\alpha^2 \phi^2+\phi'^2}^{p-2}\phi')(\gamma_-).
\end{equation}
For $\gamma\in \bigcup_{i\in\mathbb N}\{4^{-i}\}$, we directly observe that
$$
(\omega_\theta \sqrt{\alpha^2 \phi^2+\phi'^2}^{p-2}\phi')(\gamma_+)=0>(\omega_\theta \sqrt{\alpha^2 \phi^2+\phi'^2}^{p-2}\phi')(\gamma_-).
$$
Moreover, the definition of $\eta_i$ via \eqref{def:etai} ensures that \eqref{ineq:jumpcondition} holds as an equality for all $\gamma\in \bigcup_{i\in\mathbb N, i\geq j}\{\frac12 4^{-i}\}$ which finishes the argument.
%

\PfStep{pf3}{pf3s3} Let $1<p<\infty$ and $\theta\in[0,1]$ be such that $(1-\theta)p<d-1$. Let $\alpha\geq\alpha_0$ and $\rho=\rho(\alpha,d,p)\in(0,1)$ be as in Step~2. We show that $v$ is a weak subsolution on $\Omega_\rho:=B_1\cap \{|x'|<\rho\}$.

We follow a similar reasoning as in \cite{FSS98}. For $k\in\mathbb N$, let $\psi_k\in C^1(\R;[0,1])$ be a cut-off function satisfying 
$$\psi_k=0\quad\mbox{on $[0,\frac12 4^{-k}]$},\quad \psi_k\equiv1\quad\mbox{on $[4^{-k},1]$},\quad \|\psi'_k\|_{L^\infty(0,1)}\leq 4^{k+1}$$
and we define $\varphi_k\in C^1(B_1)$ by $\varphi_k(x)=\psi_k(|x'|)$. For every $\eta\in C_c^1(\Omega_\rho)$ with $\eta\geq0$, we have
\begin{align}\label{est:step3:cutoff2}
\int_{\Omega_\rho}\lambda_\theta |\nabla v|^{p-2}\nabla v\cdot \nabla \phi\,dx=&\int_{\Omega_\rho}\lambda_\theta |\nabla v|^{p-2}\nabla v\cdot (\nabla ((1-\varphi_k)\eta)+\nabla (\varphi_k \eta))\,dx\notag\\
\leq&\int_{\Omega_\rho}\lambda_\theta |\nabla v|^{p-2}\nabla v\cdot \nabla ((1-\varphi_k)\eta)\,dx,
\end{align}
where we use that $0\leq \varphi_k \eta\in C_c^1(\Omega_\rho\setminus \Omega_{4^{-k-1}}) $ and that by Step~2 $v$ is a subsolution on $ \Omega_\rho\setminus \Omega_{\delta}$ for every $\delta\in(0,\rho)$. It remains to show that the integral on the right-hand side in \eqref{est:step3:cutoff2} vanishes as $k\to\infty$. Note that $0\leq 1-\varphi_k\leq1 $ and $1-\varphi_k\equiv0$ on $\Omega_\rho\setminus \Omega_{4^{-k}}$. Hence, with help of the product rule, we obtain
\begin{align*}
\biggl|\int_{\Omega_\rho}\lambda_\theta |\nabla v|^{p-2}\nabla v\cdot \nabla ((1-\varphi_k)\eta)\,dx\biggr|\leq\int_{\Omega_{4^{-k}}}\lambda_\theta |\nabla v|^{p-1}|\nabla \eta|\,dx+\int_{\Omega_{\rho}}\eta\lambda_\theta |\nabla v|^{p-2}|\nabla v\cdot \nabla \varphi_k|\,dx.
\end{align*}
By dominated convergence, the first term on the right-hand side converges to zero as $k$ tends to $\infty$ (recall that we showed in Step~1 that $\lambda_\theta |\nabla v|^p\in L^1(B_1)$). To estimate the remaining integral we use $|\nabla v\cdot \nabla \varphi_k|=|\phi'||\nabla \varphi_k|e^{\alpha x_1}\leq C4^{k+1}|\phi'|$ for some $C=C(\alpha)>0$ on the set $\{|x'|\in(\frac12 4^{-k},4^{-k})\}$ and $\nabla v\cdot \nabla \varphi_k=0$ otherwise. Hence, we have that $|\nabla v|^{p-2}|\nabla v\cdot \nabla \varphi_k|\leq C4^{k+1}|x'|^{-(p-1)}$ on $\{|x'|\in(\frac12 4^{-k},4^{-k})\}$ and thus we obtain (using $\lambda_\theta=(k+1)^{\theta(p-1)}(2|x'|)^{p\theta}$ on $\{|x'|\in(\frac12 4^{-k},4^{-k})\}$, see \eqref{def:omegatheta})
\begin{align*}
&\int_{\Omega_{\rho}}\eta\lambda_\theta |\nabla v|^{p-2}|\nabla v\cdot \nabla \varphi_k|\,dx\\
\leq& C\|\eta\|_{L^\infty(B_1)}4^{k+1}(k+1)^{\theta(p-1)}\int_{\frac12 4^{-k}}^{4^{-k}}r^{-(p-1)}r^{p\theta} r^{d-2}\,dr\\
=&C\|\eta\|_{L^\infty(B_1)}4^{k+1}(k+1)^{\theta(p-1)}\frac1{d-p(1-\theta)}4^{-k(d-p(1-\theta))}\biggl(1-2^{-(d-p(1-\theta)}\biggr)\stackrel{k\to\infty}\to0,
\end{align*}
where we use $p(1-\theta)<d-1$ the assumption and thus $d-p(1-\theta)>1$.

\PfStep{pf3}{pf3s4} Conclusion.

\substep{4.1} We consider the case $1+\frac1{d-2}<p<\infty$. Let $s>1$ and $t>\frac1{p-1}$ be such that $\frac1s+\frac1t=\frac{p}{d-1}$ and $\frac{t}{t+1}p<d-1$. We claim that there exist $0\leq \lambda\in L^s(B_1)$ with $\lambda^{-1}\in L^t(B_1)$ and an unbounded weak subsolution to \eqref{eq:negative}. We set $\theta=\frac1t \frac{d-1}p$ and observe that $\frac1s+\frac1t=\frac{p}{d-1}$ implies $\theta\in[0,1]$ and $1-\theta=\frac1s\frac{d-1}p$. Moreover, the restriction $\frac{t}{t+1}p<d-1$ in the form $p<(1+\frac1t)(d-1)$ ensures 
$$
(1-\theta)p=(1-\frac1t\frac{d-1}p)p=(p-\frac1t(d-1))<d-1.
$$ 
Hence, in view of Steps~1--3, there exist the function $v$ defined in \eqref{def:v} with $\alpha=\alpha_0=\alpha_0(p,d)\geq1$ such that $v$ is an unbounded weak subsolution to
$$
-\nabla \cdot (\lambda_\theta |\nabla v|^{p-2}\nabla v)=0\qquad\mbox{in $B(0,\rho)$ with $\rho=\rho(d,p)\in(0,1]$,}
$$
where $\lambda_\theta(x)=\omega_\theta(|x'|)$, \textit{cf.} \eqref{def:omegatheta}. Appealing to \eqref{estabove:omegatheta}, we have that there exists $C=C(d,p)>0$ such that
\begin{align*}
\|\lambda_\theta\|_{L^s(B_1)}\leq& C\biggl(\int_0^1(r^{-p}\log(2/r)^{-(p-1)})^\frac{d-1}pr^{d-2}\,dr\biggr)^{\frac{1}{s}}\\
=&C\biggl(\int_0^1r^{-1}\log(2/r)^{-(1-\frac1p)(d-1)}\,dr\biggr)^{\frac{1}{s}}<\infty
\end{align*}
where we use that $p>1+\frac1{d-2}$ implies $(1-\frac1p)(d-1)>1$. Similarly, we have
\begin{align*}
\|\lambda_\theta^{-1}\|_{L^t(B_1)}\leq&
C\biggl(\int_0^1r^{-1}\log(2/r)^{-(1-\frac1p)(d-1)}\,dr\biggr)^{\frac{1}{t}}<\infty.
\end{align*}
Finally, we observe that by a simple scaling argument namely considering $\tilde v(x)=v(x/\rho)$ and $\lambda(x):=\lambda_\theta(x/\rho)$ we find that $\tilde v$ is a weak subsolution to \eqref{eq:negative} in $B_1$ and $\lambda$ satisfies $\lambda\in L^s(B_1)$ and $\lambda^{-1}\in L^t(B_1)$.

\substep{4.2} We consider $1<p\leq 1+\frac{1}{d-2}$. Let $s$ and $t$ be as in the statement of the theorem. Clearly, we find $\overline s>s$ and $\overline t>t$ such that $\frac1{\overline s}+\frac1{\overline t}=\frac{p}{d-1}$. Hence, for $\lambda_\theta$ with $\theta=\frac1{\overline t} \frac{d-1}p$, we obtain as in Substep~4.1, an unbounded subsolution. It remains to check if $\lambda_\theta \in L^s(B_1)$ and $\lambda^{-1}\in L^t(B_1)$. By construction, we have $1-\theta=\frac1{\overline s} \frac{d-1}p$ and thus
\begin{align*}
\|\lambda_\theta\|_{L^s(B_1)}\leq& C\biggl(\int_0^1(r^{-p}\log(2/r)^{-(p-1)})^{\frac{d-1}p\frac{s}{\overline s}}r^{d-2}\,dr\biggr)^{\frac{1}{s}}\\
=&C\biggl(\int_0^1r^{-(d-1)\frac{s}{\overline s}+d-2}\log(2/r)^{-(1-\frac1p)(d-1)\frac{s}{\overline s}}\,dr\biggr)^{\frac{1}{s}}<\infty,
\end{align*}
where we use $s/\overline s<1$ and thus $-(d-1)\frac{s}{\overline s}+d-2>-1$. A similar argument shows $\lambda_\theta^{-1}\in L^t(B_1)$ which finishes the argument. 
\end{proof}

\appendix

\section{Proof of Lemma~\ref{lm1}}

\begin{proof}[Proof of Lemma~\ref{lm1}]
As a starting point we use \cite[Lemma 2.1]{HS19}, which states for any $\delta \in (0,1]$
 \begin{equation}\nonumber
  J(\rho,\sigma,v) \le (\sigma - \rho)^{-(p-1+\frac 1\delta)} 
  \biggl( \int_{\rho}^\sigma \biggl ( \int_{S_r} \mu |v|^p \textrm{ d}\mathcal{H}^{d-1} \biggr)^{\delta} \dr \biggr)^{\frac 1\delta}.
 \end{equation}
 With this at hand, we proceed in analogy to the Step 2 of Proof of \cite[Lemma 2.1]{BS19a}: 
 
 Observe that the assumption $s > 1$ implies $s_* \in [1,d-1)$. To estimate the right-hand side, on each sphere we will use ``scale-invariant'' Sobolev inequality with $\alpha := s_*$ in the form
 \begin{equation*}
  \biggl( \int_{S_r} |\phi|^{\alpha^*} \biggr)^{\frac 1 {\alpha^*}} \le c\biggl( \biggl( \int_{S_r} |\nabla \phi|^\alpha \biggr)^{\frac 1 \alpha} + \frac 1 r \biggl( \int_{S_r} |\phi|^\alpha \biggr)^\frac{1}{\alpha} \biggr),
 \end{equation*}
 which holds with $c=c(d,\alpha)$ with $1 \le \alpha < d-1$, $\frac{1}{\alpha^*} = \frac 1\alpha - \frac 1{d-1}$ and any $r > 0$. Moreover, observe that by Jensen inequality the previous estimate holds also if we change the exponent $\alpha^*$ on the l.h.s. to a smaller exponent $\alpha' \in [1,\alpha^*)$, while picking up a dimensional factor of $|S_r|^{\frac{1}{\alpha'}-\frac{1}{\alpha^*}}$. Since by assumption $r \in (\rho,\sigma) \subset [\frac 12, 2]$, we can hide this factor into the constant $c$ on the r.h.s. 
 
 
 The definition of $s_*$ implies that for $\alpha=s_*$ holds $\frac{ps}{s-1}\leq \alpha^*$. Hence, for any $\delta \in (0,1]$ we estimate
 \begin{align}\nonumber
 \biggl( \int_{\rho}^\sigma \biggl ( \int_{S_r} \mu |v|^p \biggr)^{\delta} \dr \biggr)^{\frac 1\delta} 
 &\le 
 \biggl( \int_{\rho}^\sigma \biggl ( \int_{S_r} \mu^s  \biggr)^{\frac \delta s}  
 \biggl ( \int_{S_r} |v|^{p\frac{s}{s-1}}  \biggr)^{\delta \frac{s-1}{s}} \dr \biggr)^{\frac 1\delta} \\
 \nonumber
 &\le c \biggl( \int_{\rho}^\sigma \biggl ( \int_{S_r} \mu^s  \biggr)^{\frac \delta s}  
 \biggl[
 \biggl ( \int_{S_r} |\nabla v|^{s_*}  \biggr)^{\frac{p\delta}{s_*}} +
 \frac 1{r^{p \delta}} \biggl ( \int_{S_r} |v|^{s_*}  \biggr)^{\frac{p\delta}{s_*}}  
 \biggr]
 \dr \biggr)^{\frac 1\delta},
 \end{align}
 with $s_*$ defined above. To be able to apply H\"older inequality in $r$ to get two bulk integrals, we require $\frac{\delta}{s} + \frac{p \delta}{s_*} = 1$. By choosing $\delta = (1+\frac p{d-1})^{-1} \in (0,1)$ in the case $s_* > 1$ and $\delta := (\frac 1s + p)^{-1}$ if $s_* = 1$, we obtain 
 \begin{align}\nonumber
 J((\rho,\sigma,v) \le 
 \frac{c}{(\sigma-\rho)^{\frac{pd}{d-1}}} 
 \biggl( \int_{B_\sigma \setminus B_\rho} \mu^s \biggr)^{\frac 1s}  
 \biggl[
 \biggl ( \int_{B_\sigma \setminus B_\rho} |\nabla v|^{s_*} \biggr)^{\frac{p}{s_*}} +
 \frac{1}{\rho^p} \biggl ( \int_{B_\sigma \setminus B_\rho} |v|^{s_*}  \biggr)^{\frac{p}{s_*}}  
 \biggr] 
 \end{align}
 Observe that in the latter case of $s_*=1$ and $\delta = (\frac 1s + p)$ the correct prefactor is actually $c (\sigma-\rho)^{-(2p-1+\frac 1s)}$. Nevertheless, the estimate farther holds thanks to $2p - 1 + \frac 1s \ge \frac{pd}{d-1}$, which in turn is equivalent to $1 \le \frac 1p (1-\frac 1s) + \frac 1{d-1}$ -- the condition which is exactly fulfilled in this case. 
\end{proof}

\section*{Acknowledgments}

PB was partially supported by the German Science Foundation DFG in context of the Emmy Noether Junior Research Group BE 5922/1-1. PB and MS thank Roberta Marziani for carefully reading parts of the manuscript.


\begin{thebibliography}{99}
\bibitem{AD} D.\ Albritton and H.\ Dong, Regularity properties of passive scalars with rough divergence-free drifts. \textit{arXiv:2107.12511 [math.AP]}.
 \bibitem{BCM18} P.\ Baroni, M.\ Colombo and G.\ Mingione, Regularity for general functionals with double phase. \textit{Calc.\ Var.\ Partial Differential Equations} {\bf 57} (2018), no.~2, Art.~62.
\bibitem{BDS20} A.\ K.\ Balci, L.\ Diening and M.\ Surnachev, New examples on Lavrentiev gap using fractals. \textit{
Calc.\ Var.\ Partial Differential Equations} {\bf 59} (2020), no. 5, Paper No. 180, 34 pp.
 \bibitem{BM20} L.\ Beck and G.\ Mingione, Lipschitz bounds and non-uniform ellipticity. \textit{Comm.\ Pure Appl.\ Math.} {\bf 73} (2020), 944--1034.
 \bibitem{BS19a} P.\ Bella and M.\ Sch\"affner, Local Boundedness and Harnack Inequality for Solutions of Linear Nonuniformly Elliptic Equations. \textit{Comm.\ Pure Appl.\ Math.} {\bf 74} (2021), no.~3, 453--477.
 \bibitem{BS19c} P.\ Bella and M.\ Sch\"affner, On the regularity of minimizers for scalar integral functionals with $(p,q)$-growth. \textit{Anal.\ PDE} {\bf 13} (2020), no.~7, 2241--2257.
\bibitem{BS22} P.\ Bella and M.\ Sch\"affner, Lipschitz bounds for integral functionals with $(p,q)$-growth conditions. \textit{Advances in Calculus of Variations}, 2022. https://doi.org/10.1515/acv-2022-0016
\bibitem{BS4} P.\ Bella and M.\ Sch\"ffner, Non-uniformly parabolic equations and applications to the random conductance model. \textit{Probab.\ Theory Related Fields} {\bf 182} (2022), no. 1-2, 353--397.
 \bibitem{BCM20} S.\ Biagi, G.\ Cupini and E.\ Mascolo, Regularity of quasi-minimizers for non-uniformly elliptic integrals. \textit{J.\ Math.\ Anal.\ Appl.} {\bf 485} (2020), no.~2, 123838, 20 pp. 
 \bibitem{BDMS} V.\ B\"ogelein, F.\ Duzaar, M.\ Marcellini and C.\ Scheven, Boundary regularity for elliptic systems with $p,q$-growth. \textit{
J.\ Math.\ Pures Appl.} (9) 159 (2022), 250--293. 
 \bibitem{CM15} M.\ Colombo and G.\ Mingione, Regularity for double phase variational problems, \textit{Arch.\ Ration.\ Mech.\ Anal.} {\bf 215} (2015), no.~2, 443--496.
 \bibitem{CMM18} G.\ Cupini, P.\ Marcellini and E.\ Mascolo, Nonuniformly elliptic energy integrals with $p,q$-growth, \textit{Nonlinear Anal.,} {\bf 177}, 312--324 (2018).
 \bibitem{CMMP21} G.\ Cupini, P.\ Marcellini, E.\ Mascolo and A.\ Passarelli di Napoli, Lipschitz regularity for degenerate elliptic integrals with $p,q$-growth. \textit{Adv.\ Calc.\ Var.} https://doi.org/10.1515/acv-2020-0120
 \bibitem{DM19} C.\ De Filippis and G.\ Mingione, On the regularity of minima of non-autonomous functionals, \textit{J. Geom. Anal.} {\bf 30} (2020), no.~2, 1584--1626.
 \bibitem{DM21} C.\ De Filippis and G.\ Mingione, Lipschitz bounds and nonautonomous integrals. \textit{Arch.\ Ration.\ Mech.\ Anal.} {\bf 242} (2021), 973--1057.
 \bibitem{DM22} C.\ De Filippis and G.\ Mingione, Nonuniformly elliptic Schauder Theory. \textit{arXiv:2201.07369}.
 \bibitem{DP22} C.\ De Filippis and M.\ Piccinini, Borderline global regularity for nonuniformly elliptic systems. \textit{arXiv:2206.15330 [math.AP]}.
\bibitem{EMM19} M.\ Eleuteri, P.\ Marcellini and E.\ Mascolo, Regularity for scalar integrals without structure conditions. \textit{Adv.\ Calc.\ Var.} {\bf 13} (2020), no. 3, 279--300. 
 \bibitem{ELM02} L.\ Esposito, F.\ Leonetti and G.\ Mingione, Regularity results for minimizers of irregular integrals with $(p,q)$ growth. \textit{Forum Math.} {\bf 14} (2002), no.~2, 245--272.
 \bibitem{ELM04} L.\ Esposito, F.\ Leonetti and G.\ Mingione, Sharp regularity for functionals with $(p,q)$ growth, \textit{J.\ Differential Equations} {\bf 204} (2004), no.~1, 5--55.
 \bibitem{G87} M.\ Giaquinta, Growth conditions and regularity, a counterexample, \textit{Manuscripta Math.} {\bf 59} (1987), no.~2, 245--248. 
 \bibitem{GT} D.\ Gilbarg and N.\ Trudinger, \textit{Elliptic partial differential equations of second order}, Springer, 1998. 
 \bibitem{FMM04} I.\ Fonseca, J.\ Mal\'y and G.\ Mingione, Scalar minimizers with fractal singular sets. \textit{Arch.\ Ration.\ Mech.\ Anal.} {\bf 172} (2004), no.~2, 295--307. 
 \bibitem{FSS98} B.\ Franchi, R.\ Serapioni and F.\ Serra Cassano, Irregular solutions of linear degenerate elliptic equations, \textit{Potential Anal.} {\bf 9} (1998), no. 3, 201--216.
 \bibitem{HO21} P.\ H\"ast\"o and J.\ Ok, Maximal regularity for local minimizer of non-autonomous functionals. \textit{J.\ Eur.\ Math.\ Soc.\ (JEMS)} {\bf 24} (2022), no. 4, 1285--1334.
 \bibitem{HS19} J.\ Hirsch and M.\ Sch\"affner, Growth conditions and regularity, an optimal local boundedness result. \textit{Commun.\ Contemp.\ Math.} {\bf 23} (2021), no.~3, Paper No.~2050029.
\bibitem{HL} Q.\ Han and F.\ Lin, \textit{Elliptic partial differential equations}, Courant Lecture Notes in Mathematics, vol. 1, New York University, Courant Institute of Mathematical Sciences, New York. American Mathematical Society, Providence, RI (1997).
 \bibitem{LU68} O.\ Ladyzhenskaya and N.\ Ural'tseva, Linear and quasilinear elliptic equations. Leon Ehrenpreis Academic Press, New York-London xviii+495, 1968
 \bibitem{Mar89} P.\ Marcellini, Regularity of minimizers of integrals of the calculus of variations with nonstandard growth conditions, \textit{Arch.\ Rational Mech.\ Anal.} {\bf 105} (1989), no.~3, 267--284. 
 \bibitem{Mar91} P.\ Marcellini, Regularity and existence of solutions of elliptic equations with $p,q$-growth conditions, \textit{J.\ Differential Equations} {\bf 90} (1991), no.~1, 1--30. 
 \bibitem{Min06} G.\ Mingione, Regularity of minima: an invitation to the dark side of the calculus of variations, \textit{Appl.\ Math.} {\bf 51} (2006), no.~4,  355--426.
 \bibitem{MR21} G.\ Mingione and V.\ R\v{a}dulescu, Recent developments in problems with nonstandard growth and nonuniform ellipticity. \textit{J.\ Math.\ Anal. Appl.} {\bf 501} (2021), no. 1, Paper No. 125197, 41 pp.

 
 
 \bibitem{MS68} M.\ K.\ V.\ Murthy and G.\ Stampacchia, Boundary value problems for some degenerate-elliptic operators, \textit{Ann.\ Mat.\ Pura Appl.\ (4)} {\bf 80} (1968), 1--122. 
 \bibitem{Schwarzmannthesis} A.\ Schwarzmann, Optimal Boundedness Results for Degenerate Elliptic Equations, Thesis TU Dortmund, 2020.
 \bibitem{T67} N.\ S.\ Trudinger, On Harnack type inequalities and their application to quasilinear elliptic equations. \textit{Comm. Pure Appl. Math.} {\bf 20} (1967), 721--747. 
 \bibitem{T71} N.\ S.\ Trudinger, On the regularity of generalized solutions of linear, non-uniformly elliptic equations, \textit{Arch.\ Rational Mech.\ Anal.} {\bf 42} (1971), 50--62.
 \bibitem{T73} N.\ Trudinger, Linear elliptic operators with measurable coefficients, \textit{Ann.\ Scuola Norm.\ Sup.\ Pisa,} {\bf 27}, 265--308 (1973).
 \bibitem{zhang} X.\ Zhang, Maximum principle for non-uniformly parabolic equations and applications. \textit{arXiv:2012.05026 [math.AP]}.

\end{thebibliography}
\end{document}